\documentclass[11pt]{article}

\usepackage[margin=1in]{geometry}
\usepackage{amsmath,amssymb,amsthm}
\usepackage{enumitem}
\usepackage{xcolor}
\usepackage{hyperref}

\newcommand{\Q}{\mathbb{Q}}
\newcommand{\R}{\mathbb{R}}

\newtheorem{theorem}{Theorem}[section]
\newtheorem{lemma}[theorem]{Lemma}
\newtheorem{corollary}[theorem]{Corollary}

\title{A Proof of a Conjecture of M\'oricz and Nagy on Rational-Value Sums}

\author{Jing Huang\footnote{E-mail address: {\it
jhuangmath@gzhu.edu.cn (J. Huang).
}}
}

\date{\small
School of Mathematics and Information Science,
Guangzhou University, Guangzhou 510006, China
}

\begin{document}
\maketitle

\begin{abstract}
M\'oricz and Nagy introduced the problem of maximizing 
the number of $r$-element subsets with  rational 
sums in an $n$-element set of irrational numbers, 
and showed that it is equivalent to an extremal 
zero-sum problem. They determined the exact maximum 
in several cases. For the remaining range, 
they presented an explicit construction of an $n$-element
set of irrational numbers containing exactly
$m\binom{n-m}{r-1}$ such subsets,
where $m=\lfloor n/r\rfloor$.
They conjectured that this construction 
is always optimal for any 
$1<r<n$.
In this paper, we confirm that conjecture. 
Our proof combines an order-theoretic 
antichain argument for zero-sum subsets 
with a sharp maximization of the resulting 
binomial expressions. As a consequence, 
we determine exactly the maximum number 
of $r$-term zero-sum subsequences in sequences 
of $n$ nonzero integers.


\end{abstract}

\section{Introduction}  
For an $n$-element set $A\subseteq \R\setminus\Q$ 
and an integer $r$ with $1<r<n$, let
\[
\mathcal{H}(A,r)=\bigl\{B\subseteq A: |B|=r \text{ and } \sum_{a\in B} a\in \Q\bigr\},
\]
and define
\[
h(n,r)=\max\bigl\{|\mathcal{H}(A,r)|: A\subseteq \R\setminus\Q,\ |A|=n\bigr\}.
\]
The problem of determining $h(n,r)$ was 
introduced by M\'oricz and Nagy \cite{MN}. 
In its original formulation, this is an extremal additive 
problem: one seeks the maximum number of $r$-subsets that satisfy a prescribed arithmetic condition. 
%In this vein, it aligns with the
%Manickam-Mikl\'{o}s-Singhi line of research and related extremal questions %concerning subset sums  \cite{AHS,MMS1,MMS2,Pokrovskiy}.

However, the decisive point for the present problem 
is that this irrational-number formulation is only 
the route by which the question is posed, not the 
language in which it is ultimately solved. 
One of the key observations of
M\'oricz and Nagy \cite{MN} is that,
after applying a suitable $\Q$-linear map
whose kernel is exactly $\Q$, the condition
that a subset sum be rational is
transformed into the condition that a
corresponding sum be zero.
Accordingly, the present problem is most naturally
studied as an extremal zero-sum problem.
This point of view is sufficient for the proof of our main result
(Theorem~\ref{theorem} below), and it will also yield, as a consequence,
the same extremal value for sequences of nonzero rational
numbers and hence, after clearing denominators,
for sequences of nonzero integers.
Accordingly,
although the problem was proposed by analogy with
extremal set-sum questions, the mathematics that
governs it is zero-sum theory.

Much of zero-sum theory concerns direct and inverse problems
\cite{Caro,GaoGeroldinger,Geroldinger2009}.
On the direct side, one seeks length thresholds 
that force the existence of a zero-sum subsequence; 
the Erd\H{o}s--Ginzburg--Ziv theorem \cite{EGZ} 
and the Davenport constant are the standard prototypes. 
On the inverse side, one studies the structure of extremal sequences which contain no such zero-sum subsequences;
the structural characterization of extremal EGZ 
sequences \cite{Bialostocki1992} and the Savchev--Chen theorem
\cite{SavchevChen2007} 
are the prime examples here.

From this perspective, the problem considered here is atypical. 
It neither asks for a threshold that guarantees a zero-sum 
subsequence nor for the structure of zero-sum free sequences. 
Instead, it reverses the usual lower-bound viewpoint 
and asks for an extremal upper bound: 
how many fixed-length zero-sum subsequences 
can be contained in a sequence of nonzero terms? 
In this sense, it serves as a natural complement to the mainstream zero-sum literature. A particularly close counting precedent 
is a conjecture of Bialostocki \cite{B}---motivated by the 
Erd\H{o}s--Ginzburg--Ziv theorem---concerning the minimum number of zero-sum subsequences of length $m$ (modulo $m$) in an integer sequence; this line of research was extensively developed and resolved in various cases
by Bialostocki and collaborators, 
Kisin, F\"uredi and Kleitman, and Grynkiewicz \cite{Bialostocki2003,FK,Grynkiewicz,Kisin}. 
Our problem points in the opposite direction: 
it seeks the maximum rather than the minimum.

M\'oricz and Nagy \cite{MN} determined 
$h(n,2)$, $h(n,3)$, and $h(n,r)$ for $r\ge n/2$. 
For the remaining uniform range, 
they proposed a two-coset construction. Write
$
m=\left\lfloor \frac{n}{r}\right\rfloor,
$
fix an irrational number $\alpha$, and let $A$ consist of $m$ distinct elements from $\Q-(r-1)\alpha$ and $n-m$ distinct elements from $\Q+\alpha$. Then an $r$-subset of $A$ has rational sum if and only if it contains exactly one element from $\Q-(r-1)\alpha$ and $r-1$ elements from $\Q+\alpha$, so this construction yields
\[
m\binom{n-m}{r-1}
\]
rational-sum $r$-subsets \cite[Construction~5.1 and Proposition~5.2]{MN}. In \cite[Conjecture~5.3]{MN}, they asked whether this lower bound is always best possible.

The purpose of this paper is to prove that conjecture. 
Equivalently, we determine the maximum number of $r$-term zero-sum subsequences in length-$n$ sequences of nonzero rational numbers, and hence also in length-$n$ sequences of nonzero integers. Thus the result supplies an exact extremal companion to the existence and lower-bound problems that dominate zero-sum theory. To establish the exact maximum, our proof builds upon the rational-to-zero-sum reduction introduced in \cite{MN}. The core novelty lies in capturing the inherent interplay between positive and negative terms within a zero-sum subset. By casting this natural constraint as a partial order on a product poset, we deploy an antichain argument, followed by a sharp optimization of the resulting binomial expressions to yield the desired tight upper bound.
More precisely, we have the following results. 
\begin{theorem}
\label{theorem}
Let $n>1$ be a positive integer, $1<r<n$, and put
$
m=\left\lfloor \frac{n}{r}\right\rfloor.
$
Then every $n$-element set $A\subseteq \R\setminus\Q$ satisfies
\[
|\mathcal{H}(A,r)|\le m\binom{n-m}{r-1}.
\]
We further have 
\[
h(n,r) = m\binom{n-m}{r-1}.
\]
\end{theorem}
The following   corollary, which completely solves Problem
1.2 of \cite{MN}, will be deduced after the proof of Theorem \ref{theorem}.

\begin{corollary}
\label{cor:problem12}
Let $n>1$ be a positive integer and $1<r<n$. For a sequence
$\mathbf{x}=(x_1,\dots,x_n)$ of nonzero rational numbers, let
\[
Z_{\mathbf{x}}(r)=\bigl\{I\subseteq\{1,\ldots,n\}: |I|=r \text{ and } \sum_{i\in I}x_i=0\bigr\}.
\]
Define
\[
z_{\mathbb{Q}}(n,r)=\max\bigl\{|Z_{\mathbf{x}}(r)|: \mathbf{x}\in (\mathbb{Q}\setminus\{0\})^n\bigr\}
\]
and
\[
z_{\mathbb{Z}}(n,r)=\max\bigl\{|Z_{\mathbf{x}}(r)|: \mathbf{x}\in (\mathbb{Z}\setminus\{0\})^n\bigr\}.
\]
If $m=\lfloor n/r\rfloor$, then
\[
z_{\mathbb{Q}}(n,r)=z_{\mathbb{Z}}(n,r)=m\binom{n-m}{r-1}.
\]
Furthermore, equality is attained by the integer sequence
\[
(\underbrace{-(r-1),\dots,-(r-1)}_{m\text{ times}},
\underbrace{1,\dots,1}_{n-m\text{ times}}).
\]
\end{corollary}
The remainder of the paper is mainly
devoted to the proof of Theorem~\ref{theorem}.

\section{Proof of Theorem \ref{theorem} and Corollary
\ref{cor:problem12}}
To establish Theorem \ref{theorem}, our proof is
divided into three intermediate lemmas.
Before proceeding to the details,
let us fix some standard notation and conventions
used throughout this section.
For any positive integer $n$,
we denote the set $\{1, \dots, n\}$ by $[n]$.
Furthermore, we adopt the extended definition
for binomial coefficients, setting
$\binom{n}{k} = 0$ whenever $k<0$ or $k>n$.

The following lemma is essentially \cite[Lemma~2.1]{MN};
we include a proof for completeness.

\begin{lemma}
\label{linear}
Let $V\subseteq \R$ be a finite-dimensional
vector space over $\Q$ such that $\Q\subseteq V$.
Then there exists a $\Q$-linear
map $\varphi:V\to \R$ with $\ker \varphi=\Q$.
\end{lemma}

\begin{proof}
Choose a basis $1,\beta_1,\dots,\beta_d$ of $V$ over $\Q$. Choose real numbers $\gamma_1,\dots,\gamma_d$ that are linearly independent over $\Q$. Define $\varphi$ on the basis by
\[
\varphi(1)=0,\qquad \varphi(\beta_i)=\gamma_i\quad (1\le i\le d),
\]
and extend $\Q$-linearly to all of $V$. Every rational number lies in the kernel because $\varphi(1)=0$. Conversely, if
\[
q+\sum_{i=1}^d c_i\beta_i\in \ker\varphi
\qquad (q,c_1,\dots,c_d\in \Q),
\]
then
$
\sum_{i=1}^d c_i\gamma_i=0.
$
By the $\Q$-linear independence of $\gamma_1,\dots,\gamma_d$, all coefficients $c_i$ are zero. Hence the kernel is exactly $\Q$.
\end{proof}

By virtue of the core homomorphism
established in Lemma \ref{linear},
our original problem is essentially reduced to
bounding the number of zero-sum $r$-subsets.
The proof of our next fundamental bound relies on
order-theoretic techniques.
Before stating the lemma, we briefly recall some
standard terminology.
A partially ordered set, or \emph{poset},
is a pair $(X, \preceq)$ consisting of a set $X$
and a binary relation $\preceq$ that is reflexive,
antisymmetric, and transitive.
Two elements $x, y \in X$ are said to be \emph{comparable} if
either $x \preceq y$ or $y \preceq x$.
A subset $C \subseteq X$ is called a \emph{chain} if
every pair of elements in $C$ is comparable.
Conversely, a subset $\mathcal{A} \subseteq X$ is
an \emph{antichain} if no two distinct elements
in $\mathcal{A}$ are comparable.
With these notions in place, we can formulate our second lemma.
\begin{lemma}
\label{antichain}
Let $x_1,\dots,x_n$ be nonzero real numbers.
For $1<r<n$, let
\[
\mathcal{Z}_r(x_1,\dots,x_n)
=\bigl\{I\subseteq [n]: |I|=r \text{ and }
\sum_{i\in I}x_i=0\bigr\}.
\]
Let
\[
P=\{i\in [n]: x_i>0\},\qquad N=\{i\in [n]: x_i<0\},
\]
and put $p=|P|$, $q=|N|$. Then
\[
|\mathcal{Z}_r(x_1,\dots,x_n)|\le
\max_{1\le t\le r-1}\binom{q}{t}\binom{p}{r-t}.
\]
\end{lemma}

\begin{proof}
If $p=0$ or $q=0$, then every
$r$-subset of $[n]$ has strictly
positive or strictly negative sum, giving
$
\mathcal{Z}_r(x_1,\dots,x_n)=\emptyset.
$
Hence the claimed bound is trivial in this case.
Therefore we may assume from now on that
$
p\ge 1$ and $q\ge 1$.
Define
\[
\mathcal{X}=\bigl\{(S,T): S\subseteq
P,\ T\subseteq N,\ |S|+|T|=r\bigr\}.
\]
Give $\mathcal{X}$ the partial order
\[
(S,T)\preceq (S',T')\quad\Longleftrightarrow
\quad S\subseteq S' \text{ and } T\supseteq T'.
\]
If $(S,T)\prec (S',T')$, then $S\subseteq S'$ and $T\supseteq T'$ with at least one inclusion strict. Since
\[
|S|+|T|=|S'|+|T'|=r,
\]
this forces both $S'\setminus S$ and $T\setminus T'$ to be nonempty. Hence
\[
\sum_{i\in S'\cup T'}x_i-\sum_{i\in S\cup T}x_i
=\sum_{i\in S'\setminus S}x_i-\sum_{i\in T\setminus T'}x_i>0,
\]
because every term in the first sum is positive and every term in the second sum is negative.
For $I\in \mathcal{Z}_r(x_1,\dots,x_n)$, write
\[
P(I)=I\cap P,\qquad N(I)=I\cap N.
\]
Then $|P(I)|+|N(I)|=r$ and $1\le |N(I)|\le r-1$. Moreover,
\[
I=P(I)\cup N(I),
\]
so the map $I\mapsto (P(I),N(I))$ is injective. Therefore the family
\[
\mathcal{A}=\bigl\{(P(I),N(I)):
I\in \mathcal{Z}_r(x_1,\dots,x_n)\bigr\}
\]
is an antichain in $\mathcal{X}$, and
\[
|\mathcal{A}|=|\mathcal{Z}_r(x_1,\dots,x_n)|.
\]
Indeed, if two distinct members of $\mathcal{A}$ were comparable, the corresponding zero-sum subsets would have different sums by the strict-positivity observation above, which is impossible.

To bound the size of $\mathcal{A}$, we formalize a
probabilistic model. Let $S_P$ and $S_N$ denote
the symmetric groups of all permutations of the
sets $P$ and $N$, respectively.
We define our sample space as the
Cartesian product $\Omega= S_P \times S_N.$
Since $|P|=p$ and $|N|=q$, the size of the
sample space is $|\Omega| = p!q!$.
We equip $\Omega$ with the uniform probability
measure $\mathbb{P}$, such that for every
elementary event $\omega = (\pi, \sigma) \in \Omega$, we have
\[
\mathbb{P}(\{\omega\}) = \frac{1}{p!q!}.
\]
Set $s_-=\max\{0,r-q\}$ and $s_+=\min\{r,p\}$.
For each $\omega = (\pi, \sigma) \in \Omega$
and each integer $s$ with $s_-\le s\le s_+$, define
\[
C_s(\omega)=\bigl(\{\pi_1,\dots,\pi_s\},
\{\sigma_1,\dots,\sigma_{r-s}\}\bigr).
\]
Since $C_s(\omega) \prec C_{s+1}(\omega)$
holds by definition, the set
$\mathcal{C}(\omega)= \{C_s(\omega): s_-\le s\le s_+\}$
forms a chain in $\mathcal{X}$
for every $\omega \in \Omega$.
We define a random variable
$X: \Omega \to \mathbb{R}$ that measures the size
of the intersection between the deterministic
antichain $\mathcal{A}$ and the random chain
$\mathcal{C}(\omega)$:
\[
X(\omega)= |\mathcal{A} \cap \mathcal{C}(\omega)|.
\]
Because a chain and an antichain can intersect
in at most one element, we have
$X(\omega) \le 1$ for all $\omega \in \Omega$.
Taking the mathematical expectation of $X$ over
$\Omega$ gives
\[
\mathbb{E}[X] = \sum_{\omega \in \Omega}
X(\omega)\mathbb{P}(\{\omega\}) \le 1.
\]
For each fixed element $e = (S,T) \in \mathcal{A}$,
define the indicator random variable
$I_e: \Omega \to \{0,1\}$
such that $I_e(\omega) = 1$
if $e \in \mathcal{C}(\omega)$ and
$0$ otherwise. It follows that
$X(\omega) = \sum_{e \in \mathcal{A}} I_e(\omega)$.
By the linearity of expectation,
we obtain
\[
\mathbb{E}[X] = \sum_{e \in \mathcal{A}}
\mathbb{E}[I_e] =
\sum_{e \in \mathcal{A}}
\mathbb{P}\bigl(e \in \mathcal{C}(\omega)\bigr).
\]
We now compute the exact probability
$\mathbb{P}\bigl(e \in \mathcal{C}(\omega)\bigr)$
for a given $e=(S,T) \in \mathcal{A}$
with $|S|=s$ and $|T|=r-s$.
The event $e \in \mathcal{C}(\omega)$
occurs if and only if the first
$s$ entries of $\pi$ exactly form the set $S$,
and the first $r-s$ entries of $\sigma$ exactly
form the set $T$. The number of permutations
$\pi \in S_P$ satisfying the first condition
is $s!(p-s)!$, and the number of
permutations $\sigma \in S_N$ satisfying the second
condition is $(r-s)!(q-r+s)!$.
Thus, the number of outcomes $\omega \in \Omega$
triggering this event is exactly $s!(p-s)!(r-s)!(q-r+s)!$.
Dividing by the sample space size yields
\[
\mathbb{P}\bigl(e \in \mathcal{C}(\omega)\bigr) = \frac{s!(p-s)!(r-s)!(q-r+s)!}{p!q!}
=\frac{1}{\binom{p}{s}\binom{q}{r-s}}
=\frac{1}{\binom{p}{|S|}\binom{q}{|T|}}.
\]
Substituting this exact probability into the
expectation bound gives the inequality:
\begin{equation}
\label{sum}
\sum_{(S,T)\in \mathcal{A}}
\frac{1}{\binom{p}{|S|}\binom{q}{|T|}}\le 1.
\end{equation}
Set
\[
t_0=\max\{1,r-p\},\qquad t_1=\min\{r-1,q\}.
\]
For $t_0\le t\le t_1$, let $z_t$ denote
the number of members
$I\in \mathcal{Z}_r(x_1,\dots,x_n)$ with exactly $t$
negative elements, that is,
$|N(I)|=t$.
Grouping the sum in the left hand side of inequality
(\ref{sum})
by the parameter $t$, we obtain
\[
\sum_{t=t_0}^{t_1}\frac{z_t}{\binom{p}{r-t}\binom{q}{t}}\le 1.
\]
Consequently,
\begin{align*}
|\mathcal{Z}_r(x_1,\dots,x_n)|
&=\sum_{t=t_0}^{t_1} z_t \\
&\le \Bigl(\max_{1\le t\le r-1}\binom{q}{t}\binom{p}{r-t}\Bigr)
\sum_{t=t_0}^{t_1}\frac{z_t}{\binom{p}{r-t}\binom{q}{t}} \\
&\le \max_{1\le t\le r-1}\binom{q}{t}\binom{p}{r-t}.
\end{align*}
This proves the lemma.
\end{proof}

With Lemma \ref{antichain}, it remains to
maximize this combinatorial expression over
all possible values of $q$.
Our final lemma provides this exact algebraic maximum.
\begin{lemma}
\label{ine}
Let $1<r<n$ and let $m=\lfloor n/r\rfloor$. Then
\[
\max_{0\le q\le n}\
\max_{1\le t\le r-1}\binom{q}{t}\binom{n-q}{r-t}
= m\binom{n-m}{r-1}.
\]
\end{lemma}

\begin{proof}
Write
$
n=mr+\ell$
with $0\le \ell<r$.
Fix $t\in\{1,\dots,r-1\}$ and define
\[
F_t(q)=\binom{q}{t}\binom{n-q}{r-t}~\hbox{for~$0\le q\le n$}.
\]
For $t\le q\le n-r+t$,  we have
\[
\frac{F_t(q+1)}{F_t(q)}
=\frac{q+1}{q+1-t}\cdot \frac{n-q-r+t}{n-q}.
\]
A direct simplification shows that
\[
F_t(q+1)\ge F_t(q)
\quad\Longleftrightarrow\quad
rq\le t(n+1)-r.
\]
Set
\[
q_t=\left\lfloor \frac{t(n+1)}{r}\right\rfloor.
\]
Since $n>r$, we have $t(n+1)/r>t$, and
\[
n-r+t+1-\frac{t(n+1)}{r}
=\frac{(r-t)(n-r+1)}{r}>0.
\]
Hence
$
t\le q_t\le n-r+t.
$
Moreover, the ratio criterion implies that
$F_t$ is nondecreasing for $t\le q\le q_t-1$
and nonincreasing for $q_t\le q\le n-r+t$.
Since $\binom{q}{t}=0$ for $q<t$ and $\binom{n-q}{r-t}=0$ for $q>n-r+t$, we also have
\[
F_t(q)=0 \qquad \text{for } q<t \text{ or } q>n-r+t.
\]
Hence $q_t$ is a maximizer of $F_t$
on the whole interval $0\le q\le n$.
Since $n=mr+\ell$ and
$q_t=\lfloor\frac{t(n+1)}{r}\rfloor,$
we have
$
q_t=tm+a,
$
where $a=\lfloor \frac{t(\ell+1)}{r}\rfloor.$
Define
\[
b=\ell-a,\qquad B=(r-t)m+b=n-q_t.
\]
Since $t\le r-1$, we have $t(\ell+1)/r<\ell+1$, so $a\le \ell$ and hence $b\ge 0$.
Also,
\[
b=\ell-a
=\left\lceil \frac{(r-t)(\ell+1)}{r}\right\rceil-1
\le r-t-1.
\]
The equality above follows from
the identity $\lfloor x\rfloor+\lceil c-x\rceil=c$ for every real number $x$ and every integer $c$.
We claim that
\[
F_t(q_t)=\binom{tm+a}{t}\binom{B}{r-t}\le m\binom{n-m}{r-1}.
\]
To this end, let $u=(t-1)m+a.$
Since $tm+a=u+m$, Vandermonde's identity gives
\[
\binom{tm+a}{t}=\binom{u+m}{t}=\sum_{k=0}^{t}\binom{m}{k}\binom{u}{t-k}.
\]
Multiplying by $\binom{B}{r-t}$ yields
\[
\binom{tm+a}{t}\binom{B}{r-t}
=\sum_{k=0}^{t}\binom{m}{k}\binom{u}{t-k}\binom{B}{r-t}.
\]
We bound the summands separately.
First, for $k=0$, the inequality $b\le r-t-1$ implies
\[
\frac{\binom{B}{r-t}}{\binom{B}{r-t-1}}
=\frac{B-r+t+1}{r-t}
=\frac{(r-t)(m-1)+b+1}{r-t}
\le m,
\]
yielding
\[
\binom{u}{t}\binom{B}{r-t}
\le m\binom{u}{t}\binom{B}{r-t-1}.
\]
Next, let $1\le k\le t$. We aim to prove that
\[
\binom{m}{k}\binom{B}{r-t}
\le m\binom{B}{r-t+k-1}.
\]
If $k>m$, then $\binom{m}{k}=0$, so there is nothing to prove. If $k=1$, the claim is immediate.
Hence assume that
$
2\le k\le \min\{t,m\}.
$
Since
\[
\binom{B}{r-t+k-1}
=\binom{B}{r-t}\prod_{j=1}^{k-1}\frac{B-r+t-j+1}{r-t+j},
\]
it is enough to show that for every $1\le j\le k-1$,
\begin{equation}
\label{required}
\frac{B-r+t-j+1}{r-t+j}\ge \frac{m-j}{j+1}.
\end{equation}
For such $j$ we have $j\le k-1\le m-1$,
so the right-hand side is positive. Therefore
multiplying these inequalities for $j=1,\dots,k-1$ yields
\[
\prod_{j=1}^{k-1}\frac{B-r+t-j+1}{r-t+j}
\ge \prod_{j=1}^{k-1}\frac{m-j}{j+1},
\]
and hence
\[
m\prod_{j=1}^{k-1}\frac{B-r+t-j+1}{r-t+j}
\ge m\prod_{j=1}^{k-1}\frac{m-j}{j+1}
=\binom{m}{k}.
\]
Therefore,
it remains to show that the inequality (\ref{required})
holds true for each $1\leq j\leq k-1$.
Using $B=(r-t)m+b$, we compute
\begin{align*}
&(j+1)(B-r+t-j+1)-(m-j)(r-t+j)\\
&=(j+1)\bigl((r-t)(m-1)+b-j+1\bigr)-(m-j)(r-t+j)\\
&=jm(r-t-1)+(j+1)b-(r-t)+1\\
&\ge j(r-t-1)-(r-t)+1\\
&=(j-1)(r-t-1)\ge 0.
\end{align*}
This proves   inequality (\ref{required}) and hence also
\[
\binom{m}{k}\binom{B}{r-t}
\le m\binom{B}{r-t+k-1}
\qquad (1\le k\le t).
\]
Putting the $k=0$ and $k\ge 1$ estimates together, we obtain
\begin{align*}
\binom{tm+a}{t}\binom{B}{r-t}
&\le m\sum_{k=0}^{t}\binom{u}{t-k}\binom{B}{r-t+k-1}\\
&=m\sum_{j=0}^{t}\binom{u}{j}\binom{B}{r-1-j}\\
&\le m\sum_{j=0}^{r-1}\binom{u}{j}\binom{B}{r-1-j}\\
&=m\binom{u+B}{r-1}
= m\binom{(r-1)m+\ell}{r-1}
= m\binom{n-m}{r-1},
\end{align*}
as required.
Thus, for every fixed $t$,
\[
\max_{0\le q\le n} \binom{q}{t}\binom{n-q}{r-t}
=F_t(q_t)\le m\binom{n-m}{r-1}.
\]
Taking the maximum over $t\in\{1,\dots,r-1\}$ proves
\[
\max_{0\le q\le n}\ \max_{1\le t\le r-1}\binom{q}{t}\binom{n-q}{r-t}
\le m\binom{n-m}{r-1}.
\]
The reverse inequality is attained by $q=m$ and $t=1$, because then
\[
\binom{q}{t}\binom{n-q}{r-t}=\binom{m}{1}\binom{n-m}{r-1}=m\binom{n-m}{r-1}.
\]
It follows that the equality holds, and the lemma is proved.
\end{proof}

\medskip
We are now in a position to combine the three preceding
lemmas to prove our main result.
\begin{proof}[Proof of Theorem \ref{theorem}]
Let $A=\{a_1,\dots,a_n\}$ be an
$n$-element set of irrational numbers.
Apply Lemma \ref{linear} to the finite-dimensional
$\Q$-vector space
\[
V=\langle A\cup\{1\}\rangle_{\Q}\subseteq \R.
\]
Choose $\varphi:V\to \R$ with $\ker\varphi=\Q$, and put
\[
x_i=\varphi(a_i)\qquad (1\le i\le n).
\]
Since each $a_i$ is irrational and $\ker\varphi=\Q$, all $x_i$ are nonzero.
Moreover, for every $r$-subset $I\subseteq [n]$ we have
\[
\sum_{i\in I} a_i\in \Q
\quad\Longleftrightarrow\quad
\varphi\!\left(\sum_{i\in I} a_i\right)=0
\quad\Longleftrightarrow\quad
\sum_{i\in I} x_i=0.
\]
Therefore
\[
|\mathcal{H}(A,r)|=|\mathcal{Z}_r(x_1,\dots,x_n)|.
\]
By Lemmas \ref{antichain} and \ref{ine}, we have
\[
|\mathcal{H}(A,r)|
=|\mathcal{Z}_r(x_1,\dots,x_n)|
\le m\binom{n-m}{r-1}.
\]
This proves the upper bound.
For the lower bound, fix an irrational number $\alpha$.
Choose distinct rationals
\[
u_1,\dots,u_m,\qquad v_1,\dots,v_{n-m},
\]
and set
\[
b_i=u_i-(r-1)\alpha\quad (1\le i\le m),
\qquad
c_j=v_j+\alpha\quad (1\le j\le n-m).
\]
Then all $b_i$ and $c_j$ are irrational. They are pairwise distinct: the $b_i$'s are
distinct because the $u_i$'s are distinct, the $c_j$'s are distinct because the $v_j$'s are
distinct, and $b_i\ne c_j$ for all $i,j$, since
\[
u_i-(r-1)\alpha=v_j+\alpha
\]
would imply $u_i-v_j=r\alpha$, impossible because the left-hand side is rational and the
right-hand side is irrational. Therefore
\[
A=\{b_1,\dots,b_m,c_1,\dots,c_{n-m}\}
\]
is exactly a set produced by \cite[Construction 5.1]{MN}.
Consider an $r$-subset of $A$ containing exactly $j$ elements among $b_1,\dots,b_m$.
The coefficient of $\alpha$ in the sum of its elements is
\[
j\bigl(-(r-1)\bigr)+(r-j)\cdot 1 = r(1-j).
\]
Hence the sum is rational if and only if $j=1$. Therefore the rational-sum $r$-subsets of $A$ are exactly those formed by choosing one $b_i$ and $r-1$ of the $c_j$'s, and their number is
\[
m\binom{n-m}{r-1}.
\]
Thus the upper bound is sharp.
\end{proof}

We end this paper by presenting a proof  for Corollary
\ref{cor:problem12}.

\begin{proof}[Proof of Corollary \ref{cor:problem12}]
Let $\mathbf{x}=(x_1,\dots,x_n)\in (\Q\setminus\{0\})^n$.
By Lemmas \ref{antichain} and \ref{ine},
\[
|Z_{\mathbf{x}}(r)|\le m\binom{n-m}{r-1}.
\]
Hence
\[
z_{\mathbb{Q}}(n,r)\le m\binom{n-m}{r-1}.
\]
Since every integer sequence is also a rational sequence,
we have
$
z_{\mathbb{Z}}(n,r)\le z_{\mathbb{Q}}(n,r).
$
On the other hand, consider the integer sequence
\[
\mathbf{y}=(\underbrace{-(r-1),\dots,-(r-1)}_{m\text{ times}},
\underbrace{1,\dots,1}_{n-m\text{ times}}).
\]
If an $r$-subset of indices contains exactly $j$ negative terms, then its sum is
\[
j\bigl(-(r-1)\bigr)+(r-j)\cdot 1 = r(1-j).
\]
Therefore the sum is zero if and only if $j=1$, and thus
\[
|Z_{\mathbf{y}}(r)|=m\binom{n-m}{r-1}.
\]
Consequently,
\[
m\binom{n-m}{r-1}
\le z_{\mathbb{Z}}(n,r)
\le z_{\mathbb{Q}}(n,r)
\le m\binom{n-m}{r-1},
\]
which proves that
\[
z_{\mathbb{Q}}(n,r)=z_{\mathbb{Z}}(n,r)=m\binom{n-m}{r-1}.
\]
\end{proof}

\end{document}